\documentclass[12pt]{amsart} 
\usepackage{amsmath,amssymb,amsfonts,amsthm,amscd,amstext,amsxtra,amsopn,array,url,mathrsfs,enumerate,anysize,enumitem,mathabx} 
\marginsize{1.50cm}{2.20cm}{2.20cm}{2.20cm} 
\usepackage{times} 
\usepackage{amsmath} 
\usepackage{amssymb} 
\usepackage{amsfonts} 
\usepackage{xcolor} 
\usepackage{hyperref} 

\usepackage{amsthm}
\usepackage{tikz-cd}

\usepackage{times} 
\usepackage[fleqn]{nccmath}

\makeatletter
\@namedef{subjclassname@2020}{\textup{2020} Mathematics Subject Classification}
\makeatother

\usepackage{amsrefs} 
\newenvironment{rezabib} 
{\bibdiv\biblist\setupbib} 
{\endbiblist\endbibdiv} 

\def\setupbib{\catcode`@=\active} 
\begingroup\lccode`~=`@ 
\lowercase{\endgroup\def~}#1#{\gatherkey{#1}} 
\def\gatherkey#1#2{\gatherkeyaux{#1}#2\gatherkeyaux} 
\def\gatherkeyaux#1#2,#3\gatherkeyaux{\bib{#2}{#1}{#3}}

\usepackage{amsfonts}
\usepackage{amsmath}
\usepackage{amssymb}
\usepackage{amstext}
\usepackage{amsthm}

\usepackage{mleftright}
\renewcommand{\bar}[1]{\overline{#1}}    
\newcommand*{\propref}[1]{Proposition~\ref{#1}}

\newcommand*{\lemref}[1]{Lemma~\ref{#1}}
\newcommand*{\thmref}[1]{Theorem~\ref{#1}}

\newcommand*{\tabref}[1]{Table~\ref{#1}}
\newtheorem{theorem}{Theorem}[section]
\newtheorem{lemma}[theorem]{Lemma}

\newtheorem{proposition}[theorem]{Proposition}

\newtheorem{Remarks}[theorem]{Remarks}

\newcommand{\Leg}[2]{\genfrac{(}{)}{}{}{#1}{#2}}

\newenvironment{psmallmatrix}
  {\left(\begin{smallmatrix}}
  {\end{smallmatrix}\right)}

 {\begin{list}{}%
    {\setlength{\leftmargin}{#1}}%
  \item[]%
 }
{\end{list}}

\numberwithin{equation}{section}

\begin{document}

\begin{center} 


\title{Binary quadratic forms of odd class number}

\author{Amir Akbary}
\address{Department of Mathematics and Computer Science, University of Lethbridge, Lethbridge, Alberta, T1K 3M4}
\email{amir.akbary@uleth.ca}

\author{Yash Totani} 
\address{Department of Pure Mathematics,
University of Waterloo, Waterloo, Ontario, N2L 3G1} 
\email{ytotani@uwaterloo.ca}


\thanks{This research was supported by the NSERC discovery grant RGPIN-2021-02952.}

\date{\today}

\begin{abstract} 
Let $-D$ be a fundamental discriminant. We express the number of representations of an integer by a positive definite binary quadratic form of discriminant $-D$ with an odd class number $h(-D)$ as a rational linear expression involving the Kronecker symbol $\Leg{-D}{.}$ and the Fourier coefficients of certain cusp forms. We prove these cusp forms have eta quotient representations only if $D=23$. This provides, using theta functions, a generalization of a result of F. van der Blij from 1952 for binary quadratic forms of discriminant $-23$ to the case of forms of discriminant $-D$ with odd $h(-D)$. We also classify all the eta quotients of prime level $D$ which are half the difference of two theta functions of level $D$.

\end{abstract} 
\maketitle 
\end{center}

\section{Introduction}
Let $a(n, Q)$ be the number of representations of a positive integer $n$ by an integral quadratic form 
$$Q(x_1, x_2, \dots, x_m)=\sum_{i=1}^{m}\sum_{j=i}^{m} a_{ij} x_i x_j, ~~~a_{ij}\in \mathbb{Z}.$$ 
The problem of finding formulas for $a(n, Q)$ fascinated number theorists for many centuries. Finding exact formulas, except for some special cases, seems unlikely. However,  the asymptotics for $a(n, Q)$ is often known. An instance of an exact formula was discovered in 1952. More specifically, 
F. van der Blij, in \cite{VDB}, gave exact formulas for $a(n, Q)$ for all three equivalence classes of binary quadratic forms $Q(x, y)$ of discriminant -23 (see Section \ref{binqua} for the basics of binary quadratic forms). He proved that 
\begin{equation}
\label{x1}
     a(n,Q_0)=\frac{2}{3}\sum_{d\mid n} 
\Leg{d}{23}+\frac{4}{3}t(n)
 \end{equation} 
 and
 \begin{equation}
 \label{x2}
     a(n,Q_1)=a(n,\bar{Q}_1)=\frac{2}{3}\sum_{d\mid n} 
\Leg{d}{23}-\frac{2}{3}t(n),
\end{equation}
where $\Leg{.}{23}$ is the Legendre symbol mod 23, the quadratic forms $Q_i's$ are 
 \begin{equation}
 \label{r-forms}
  Q_0(x, y)=x^2+xy+6y^2,~
    Q_1(x, y)=2x^2+xy+3y^2,~
    \bar{Q}_1(x, y)=2x^2-xy+3y^2,
    \end{equation}
and the coefficients $t(n)$ arise out of the formal identity
\begin{equation}
\label{-23disc}
\sum_{n=1}^\infty t(n)q^n=q\prod_{n=1}^\infty (1-q^n)(1-q^{23n}).    
\end{equation}
Note that, from \eqref{-23disc}, $t(n)$ can be computed explicitly, and thus the above results can be used for the computation $a(n, Q)$ of any positive definite quadratic form $Q(x, y)$ of discriminant $-23$ for any value $n$. In \cite[Section 3]{VDB}, using \eqref{-23disc}, some specific properties of $t(n)$ are derived and among other things it is shown that $t(n)\equiv \tau(n)$ (mod 23), where $\tau(n)$ is the Ramanujan $\tau$-function.
The proof in \cite{VDB} is based on combinatorial arguments. More specifically, using Euler's pentagonal number theorem, 
it is shown that
\begin{equation*}
\sum_{n=1}^\infty t(n)q^n=\sum_{m, n= -\infty}^{+\infty} (-1)^{m+n} q^{\frac{1}{24}(6m+1)^2+\frac{23}{24} (6n+1)^2}.    
\end{equation*}
This gives a relation between the values of $t(n)$ and the number of certain solutions of the Diophantine equation $u^2+23v^2=24 n$, which from it is deduced that 
$a(n, Q_0)-a(n, Q_1)=2t(n)$. This together with $a(n, Q_1)=a(n, \bar{Q}_1)$ and  $a(n, Q_0)+a(n, Q_1)+a(n, \bar{Q}_1)= 2 \sum_{d\mid n} \Leg{d}{23}$ imply the result.

In this paper, we revisit this classical theorem from the perspective of the theory of modular forms and theta functions. 
First, we note that $-23$ is one of the 16 negative fundamental discriminants of class number 3, where the class number $h(-D)$ is the cardinality of the form class group of discriminant $-D$. Our first result generalizes van der Blij's result for the positive definite quadratic forms of discriminant $-23$ to the fundamental discriminants with odd class numbers. From now on, $q=e^{2\pi i z}$ where $z$ is in the upper half-plane.

\begin{theorem}
\label{main-1}
Let the integer $-D<0$ be a fundamental discriminant and assume that $h(-D)=2k+1$ for integer $k\geq 0$.
Then, there are $2k+1$ distinct reduced forms $Q_0(x, y)$, $Q_1(x, y)$, $\bar{Q}_1 (x, y)$, $\ldots$, $Q_k(x, y)$, and $\bar{Q}_k(x, y)$, where $Q_0(x, y)$ represents the principal form and $\bar{Q}_i(x, y)=a_i x^2-b_i xy+c_i y^2$ is the conjugate form corresponding to $Q_i(x, y)=a_i x^2+b_i xy+c_i y^2$. Let $w=6$ if $D=3$, $w=4$ if $D=4$, and $w=2$ otherwise.
Let $a(n, Q)$ be the number of representations of $n$ by the quadratic form $Q(x, y)$. The following holds:

(i) If $k=0$, we have $$a(n, Q_0)={w} \sum_{d\mid n} \Leg{-D}{d},$$
where $\Leg{-D}{.}$ is the Kronecker symbol.

(ii) For $k\geq 1$ and $1\leq r\leq k$, let
\begin{equation}
\label{fdr}
F_{D,r}(z):= \frac{1}{2} \left(\sum_{a, b\in \mathbb{Z}} q^{Q_0(a, b)}-\sum_{a, b\in \mathbb{Z}} q^{Q_r(a, b)} \right) =\sum_{n=1}^{\infty} t_r(n) q^n.
\end{equation}
Then $F_{D, r}(z)$ is a cusp form of weight $1$, level $D$, and character $\Leg{-D}{\cdot}$ with $t_r(1)=1$. Moreover,
\begin{equation}
\label{x3}
a(n, Q_0)=\frac{2}{2k+1} \sum_{d\mid n} \Leg{-D}{d}+\frac{4}{2k+1} \sum_{j=1}^{k} t_j(n),
\end{equation}
and, for $1\leq r\leq k$,
\begin{equation}
\label{x4}
a(n, Q_r)=a(n, {\bar Q}_r) =\frac{2}{2k+1} \sum_{d\mid n} \Leg{-D}{d}+\frac{4}{2k+1} \sum_{\substack{j=1\\j\neq r}}^{k}t_j(n)- \frac{4k-2}{2k+1}t_r(n).
\end{equation}
\end{theorem}

\begin{Remarks}
{\rm
(i) The van der Blij's result is a particular case of Theorem \ref{main-1}. To see this, let
$F_{23, 1}(z)$ be \eqref{fdr} for $(D, r)=(23, 1)$,
where $Q_0(x,y)$ and $Q_1(x, y)$ are forms given in \eqref{r-forms}. Then, by \thmref{main-1}, $F_{23, 1}(z)$ is a cusp form of weight $1$, level $23$, and character $\Leg{-23}{\cdot}$ with $t_1(1)=1$. Moreover, \eqref{x1} and \eqref{x2} hold for $t(n)=t_1(n)$, and that by the quadratic reciprocity $\Leg{d}{23}=\Leg{-23}{d}$.
Let 
\begin{equation}
\label{eta}
\eta(z)=q^{\frac{1}{24}}\prod_{n=1}^{\infty} (1-q^n)
\end{equation}
be the Dedekind eta function. By \cite[Theorem 1.64]{O}, $\eta(z)\eta(23z)$ is a cusp form of weight $1$, level $23$, and character $\Leg{-23}{\cdot}$. Hence, $F_{23, 1}(z)$ and $\eta(z)\eta(23z)$ are in the same space of cusp forms. Since this space is 
one dimensional (see \tabref{A}) and $t_1(1)$=1, we conclude
 \begin{equation}
 \label{t1}
 \sum_{n=0}^{\infty}t_1(n)q^n=\eta(z)\eta({23}z)=q\prod_{n=1}^\infty (1-q^n)(1-q^{23n}).  
 \end{equation}

(ii) Theorem \ref{main-1} provides an efficient way of computing the representation numbers $a(n, Q)$. In fact, by considering a basis for the space of cusp forms of weight $1$, level $D$, and character $\Leg{-D}{.}$ and computing $a(n, Q)$ for positive integers $n$ up to the dimension of this space, we can express $F_{D, r}(z)$ as a linear combination of this basis to compute $(t_r(n))$, and thus $a(n, Q)$, via the expressions \eqref{x3} and \eqref{x4}. The tables of Section \ref{appendix} provide the necessary information for performing this procedure for discriminants of class number one, three, and five. The complete list of such discriminants is available to us via effective solution of Gauss's class number problem (see \cite{Wa} for the latest result on the subject). 

%
(iii) When the class number is even, several key differences prevent a straightforward description as in Theorem \ref{main-1}. For example, 
the conjugate of a reduced form may not be a reduced form, so the number of the functions $F_{D, r}(z)$ do not have a simple expression in terms of the class number as Thorem \ref{main-1}. For instance, for $-D = -87$, which has class number 6, the reduced form $3x^2 + 3xy + 8y^2$ and its conjugate $3x^2 - 3xy + 8y^2$ lie in the same class, which leads to four functions $F_{D, r}(z)$ to consider. For $-D=-420$, which has class number 8, any reduced form and its conjugate are in the same class, which results in eight functions $F_{D, r}(z)$ to consider. 
}

\end{Remarks} 
The function $\eta(z)\eta(23z)$ appearing in \eqref{t1} is an example of an eta quotient.
For integers $N (\geq 1)$ and $r_\delta$, we say that $f(z)=\prod_{\delta\mid N} \eta(\delta z)^{r_\delta}$ is an \emph{eta quotient} of level $N$.  Since $\eta(z)$ has a product representation in terms of $q$ (see \eqref{eta}), the $q$-expansion of an eta quotient can be conveniently computed. Hence, an eta-quotient representation for $F_{D, r}(z)$ in Theorem \ref{main-1} will provide an efficient way for computing the sequence $(t_r(n))$.
Given this, a natural problem regarding Theorem \ref{main-1} is determining all values of $D$ and $r$ for which $F_{D, r}(z)$ has an eta quotient representation. Our next result settles this problem. 
\begin{proposition}
\label{main-2}
Under the conditions of Theorem \ref{main-1}, the pair $(23, 1)$ is the only pair $(D, r)$ for which $F_{D, r}(z)$ has an eta quotient representation. 
\end{proposition}
Thus, among the results provided in Theorem \ref{main-1}, van der Blij's assertion is the only one with an eta quotient representation. Since an eta quotient is non-zero on the upper half-plane, the proof of Proposition \ref{main-2} can be done by studying the zeros of $F_{D, r}(z)$ via valence formula (see \lemref{valencef} and \propref{cuspall}). However,
we prove Proposition \ref{main-2} as a corollary of a more general theorem. 
A result of Eholzer and Skoruppa \cite[Section 2]{ES} (See also \lemref{kohnen}) guarantees that $F_{D, r}(z)$ has a product representation 
\begin{equation}
\label{expansion}
F_{D, r}(z)=q\prod_{n=1}^\infty (1-q^n)^{c(n)}
\end{equation}
for some integers $c(n)$ and sufficiently small $|q|$. The exponents $c(n)$ above can be conveniently computed using \eqref{fdr}. More precisely, let $(\alpha(n))$ be the sequence given by the recurrence
$$\alpha(n)=-nt_r(n+1)-\sum_{k=1}^{n-1} t_r(n-k+1) \alpha(k)$$
with the initial condition $\alpha(1)=-t_r(2)$, where $(t_r(n))$ is the sequence given in \eqref{fdr}. Then,
$$c(n)=\frac{1}{n} \sum_{d\mid n} \alpha(d) \mu(\frac{n}{d}),$$
where $\mu(\cdot)$ is the M\"{o}bius function.
If $F_{D, r}(z)$ is an eta quotient, then the exponents $c(n)$'s are bounded. For example, in \eqref{expansion} for $F_{23, 1}(z)$, we have 
$$c(n)= \begin{cases}
2& {\rm if}~ n\equiv 0~( {\rm mod}~23),\\
1& {\rm Otherwise}.
\end{cases}$$

We prove the following.
\begin{theorem}
\label{cn}
Assume that the fundamental discriminant $-D\neq -23$ is such that $h(-D)$ is an odd integer greater than $1$. 
Then the integers $c(n)$ in the expansion \eqref{expansion}
are unbounded. More generally, for any $\alpha\in \mathbb{R}$, $$\frac{1}{n^{\alpha}}{\displaystyle{\sum_{d\mid n}} d c(d)}$$
is not bounded.
\end{theorem}
Note that Theorem \ref{cn} implies that $F_{D, r}(z)$ does not have an eta quotient representation. 
The proof of it relies on an analytic statement (Proposition \ref{unboundedlemma}),
which shows that if $F_{D, r}(z)$ has a zero on the upper half-plane then the integers $c(n)$ are unbounded. The analytic assertion is inspired by the exercise in Section 2.1.3 of Serre's monograph \cite{JPS}.  Hence, the proof of Theorem  \ref{cn} reduces to finding zeros of  $F_{D, r}(z)$ on the upper half-plane. 

The eta quotient representation of some other functions associated with the binary quadratic forms has already been studied by Schoeneberg \cite{Sch}. 
The function $\sum_{a,b\in \mathbb{Z}}^{}q^{Q(a,b)}$, called the \emph{theta function} associated to the binary quadratic form $Q(x, y)$, 
is of level $D$ if $Q(x, y)$ has discriminant $-D$.
For $Q_s(x, y)$ and $Q_r(x, y)$, two distinct reduced form of fundamental discriminant $-D$, let
\begin{equation*}
\label{fdsr}
F_{D,s,r}(z):= \frac{1}{2} \left(\sum_{a, b\in \mathbb{Z}} q^{Q_s(a, b)}-\sum_{a, b\in \mathbb{Z}} q^{Q_r(a, b)} \right)
\end{equation*}
be half the difference of two theta functions of level $D$ associated to $Q_s(x, y)$ and $Q_r(x, y)$.
Here, $s$ and $r$ are non-negative integers and $Q_{0}(x, y)$, as before, is the principal form of discriminant $-D$. Hence, $F_{D, 0, r}(z)=F_{D, r}(z)$ as defined in Theorem \ref{main-1}.
We noted that $F_{D, 0, r}(z)$ for $(D, r)\neq (23, 1)$ do not have eta-quotient representations; however, as proved in \cite{Sch}, eta quotient representations might be possible for $F_{D, s, r}(z)$ for some $s\neq 0$ and $r\neq 0$. For example for $D=47$, let $Q_{s_0}(x, y)= 2x^2+xy+6y^2$ and $Q_{r_0}(x, y)= 3x^2+xy+4y^2$. Then, it can be shown that
$$F_{47, s_0, r_0}(z)= \eta(z)\eta(47z).$$

In \cite[Satz 1, p. 179]{Sch}, Schoeneberg proves that for integers $D\equiv 23$ (mod 24), without any restriction on $h(-D)$, there are two forms $Q_{s_1}(x, y)$ and $Q_{r_1}(x, y)$ given by 
\begin{equation}
\label{pair1}
Q_{s_1}(x,y)=6x^2+xy+\frac{D+1}{24}y^2
\end{equation}
and 
\begin{equation}
\label{pair2}
Q_{r_1}(x,y)=6x^2+5xy+\frac{D+25}{24}y^2
\end{equation}
for which 
\begin{equation*}  
\label{diff}
F_{D, s_1, r_1}(z)= \eta(z)\eta(Dz).
\end{equation*}
This pair consists of reduced forms when $(D+1)/24\geq6$. When $(D+1)/24<6$, which is for only finitely many values of $D$, we will replace this pair with the corresponding pair of reduced forms. Let us call this pair the \emph{Schoeneberg pair} and denote it by $(Q_{s_0}(x, y), Q_{r_0}(x, y))$. Hence, for $(D+1)/24\geq 6$, the Schoeneberg pair $(Q_{s_0}(x, y), Q_{r_0}(x, y))$ is the pair given by \eqref{pair1} and \eqref{pair2}. The Schoeneberg pairs for $(D+1)/24<6$ are listed in Table \ref{tablea}.

 \begin{table}[ht]
 \label{t1-intro}
\centering
\begin{tabular}[t]{|c c c|}  \hline
 $-D$ & $Q_{s_0}(x, y)$ & $Q_{r_0}(x, y)$  \\ 
 \hline
 $-23$ & $x^2 + xy + 6y^2$ & $2x^2 + xy + 3y^2$  \\
 \hline
 $-47$ & $2x^2 + xy + 6y^2$ & $3x^2 + xy + 4y^2$  \\
 \hline
 $-71$ & $3x^2 + xy + 6y^2$ & $4x^2 + 3xy + 5y^2$  \\
 \hline
 $-95$ & $4x^2 + xy + 6y^2$ & $5x^2 + 5xy + 6y^2$  \\
  \hline
 $-119$ & $5x^2 + xy + 6y^2$ & $6x^2 + 5xy + 6y^2$  \\
 \hline
 \end{tabular} \caption{Schoeneberg pairs when $(D+1)/24<6$}\label{tablea}
\end{table}

The Schoeneberg pairs appear in a classification problem for eta quotients. There is a large literature on classifying the eta quotients (see \cite[Pages 197-199]{HC} for a sample of such results). For example, 
an exhaustive list of primitive eta quotients that are also holomorphic modular forms of weight $1/2$ is given by Mersmann (see \cite[Theoem 5.9.4]{HC}). Further, 
\cite[Lemma 2.2] {H} lists out all seven CM eta quotients of weight 1 that are half the difference of two theta functions of binary quadratic forms of negative discriminants, not necessarily of fundamental discriminants. 
Our last theorem gives a complete classifiction of 
the eta quotients of prime level $D$ which are half the difference of two theta functions of level $D$. It turns out that Schoeneberg pairs classify such eta quotients. 
 
\begin{theorem}
\label{Pair}
An eta quotient $E(z)$ of prime level $D$ is half the difference of two theta functions of level $D$ if and only if prime $D\equiv 23\pmod{24}$ and $E(z)=\eta(z) \eta(Dz) $. Moreover, for prime $D\equiv 23\pmod{24}$, we have
\begin{equation}
\label{rep}
\eta(z) \eta(Dz) 
= \frac{1}{2} \left(\sum_{a, b\in \mathbb{Z}} q^{Q_{s_0}(a, b)}-\sum_{a, b\in \mathbb{Z}} q^{Q_{r_0}(a, b)} \right),
\end{equation}
where $(Q_{s_0}(x, y), Q_{r_0}(x, y))$ is the Schoeneberg pair of Discriminant $-D$, and the representation \eqref{rep} is unique up to proper equivalence and conjugation of forms  $Q_{s_0}(x, y)$  and  $Q_{r_0}(x, y)$
\end{theorem}

Note that, by Theorem \ref{main-1}, for discriminants $-D$ in the above theorem $F_{D, s, r}(z)=F_{D, r}(z)-F_{D, s}(z)$ is a cusp form of weight $1$, level $D$, and character $\Leg{-D}{.}$. The proof of Theorem \ref{Pair} crucially uses a classification of eta quotients in these cusp spaces (Proposition \ref{eta-23}).

The structure of the paper is as follows. In Section \ref{binqua}, we review the basics of the theory of binary quadratic forms. We give the proof of Theorem \ref{main-1} in Section \ref{Section 3}. The proof of the analytic result (Proposition \ref{unboundedlemma}) on unboundedness of the exponent coefficients in the product expansion of a non-vanishing analytic function on the upper half-plane is given in Section \ref{Serre}. Theorem \ref{cn} is proved in Section \ref{Section 5}. A classification of the eta quotients in the space of cusp forms (Proposition \ref{eta-23}) and a proof of Theorem \ref{Pair} are given in Section \ref{Section 6}. Finally, Section \ref{appendix} details the relevant specific data for Theorem \ref{main-1} when the class number is one, three, or five. 
\medskip\par
\noindent{\bf Notations:} Throughout the paper $Q(x, y)$ is a positive definite binary quadratic form of discriminant $-D$, and $Q_0(x, y)$ is the principal form of discriminant $-D$. The class number $h(-D)$ is the
cardinality of the form class group of discriminant $-D$.
The function $a(n, Q)$ denotes the number of representations of $n$ by $Q(x, y)$. 
The upper half-plane and extended upper half-plane are denoted by $\mathbb{H}$ and ${\mathbb{H}}^{*}$. 
We set $q= e^{2\pi i z}$ for $z\in \mathbb{H}$. 
The theta function of $Q(x, y)$ is denoted by $\Theta_Q(z)$.
The order of vanishing  of a modular form $f$ at a point $z\in{\mathbb{H}}^{*}$ is denoted by $\nu_z(f)$. 
We denote the multiplicative group of two by two integer matrices with determinant one by ${\rm SL}_2(\mathbb{Z})$. The subgroup of ${\rm SL}_2(\mathbb{Z})$ consisting of matrices $\begin{psmallmatrix}a & b\\c & d\end{psmallmatrix}$ where $c\equiv 0$ (mod $N$) is denoted by $\Gamma_0(N)$. The subgroup of ${\rm SL}_2(\mathbb{Z})$ consisting of matrices $\begin{psmallmatrix}a & b\\c & d\end{psmallmatrix}$ where $c\equiv 0$ (mod $N$) and $a,d\equiv 1$ (mod $N$) is denoted by $\Gamma_1(N)$. 
We denote the Kronecker symbol attached to $-D$ by $\Leg{-D}{.}$, the space of modular forms of weight $1$, level $N$, and character $\Leg{-D}{.}$ by 
$M_1(\Gamma_0(N),\left(\frac{-D}{.}\right))$, and the corresponding cusp space by $S_1(\Gamma_0(N),\left(\frac{-D}{.}\right))$. The space of cusp forms for $\Gamma_1(N)$ is denoted by $S_1(\Gamma_1(N))$. The function $\eta(z)$ is the Dedekind eta function.

\section{Binary quadratic forms}\label{binqua}
We review some basic facts and terminology of binary quadratic forms. For a detailed description see \cite[Chapter 1]{DC}.   
Recall that a primitive positive definite form $Q(x,y)=ax^2+bxy+cy^2$ is said to be a \emph{reduced form} if 
\begin{equation}
\label{rf-conditions}
\quad|b| \leq a \leq c, \text { and } b \geq 0 \text { if either }|b|=a \text { or } a=c.
\end{equation}
Two positive definite forms $Q_1(x,y)$ and $Q_2(x,y)$ are said to be \emph{properly equivalent} if there exists integers $p,q,r$ and $s$ such that $ps-qr= 1$ and 
$Q_1(x,y)=Q_2(px+qy,rx+sy).$
Every primitive positive definite form is properly equivalent to a unique reduced form. 
Given a negative discriminant $d$, the \emph{principal form} is defined to be
$$\begin{cases}
x^2-\frac{d}{4} y^2,&{\rm if}~ d \equiv 0~({\rm mod}~4),\\
x^2+xy+\frac{1-d}{4} y^2,&{\rm if}~d \equiv 1~({\rm mod}~4).
\end{cases}$$
The following property of reduced forms, used later in the proofs, is a direct consequence of the definition of a reduced form and that the discriminant is $0$ or $1$ modulo $4$. 
\begin{lemma}\label{leastpositive}
Let $Q(x, y)=ax^2+bxy+cy^2$ be a reduced form that is not principal. Then, $a>1$.
\end{lemma}
The collection of equivalence classes of primitive positive definite forms of discriminant $d$ together with Dirichlet composition operation is a finite Abelian group.
The cardinality of this group is named the \emph{class number} $h(d)$. The principal form is the identity of this group.  The inverse of the class containing the form $ax^2+bxy+cy^2$ is the class containing the \emph{conjugate} (or \emph{opposite}) form $ax^2-bxy+cy^2$.
%
%
%
%
%
Now let $d=-D$ be a fundamental discriminant. 
The following lemma characterizes the fundamental discriminants of odd class numbers.
\begin{lemma}
\label{odd}
Suppose $h(-D)$ is odd for a fundamental discriminant $-D$. 
Then, $D$ is a prime congruent to $3$ modulo $4$, or $D=4, 8$.
\end{lemma}
\begin{proof}
This is a consequence of Gauss's fundamental theorem of genera, see for example  \cite[p. 298]{ARW}.
\end{proof}

The following lemma collects some properties of reduced forms of odd fundamental discriminant which will be used later in the proofs.

\begin{lemma}
\label{conjugate}
Suppose $h(-D)$ is odd for a fundamental discriminant $-D$. Let $Q(x, y)= ax^2+bxy+cy^2$ be a reduced form that is not principal.
Let $a(n, Q)$ be the number of representations of a positive integer $n$ by
$Q(x, y)$. The following assertions hold.

(i) We have $|b|<a<c$.

(ii) The conjugate form $\bar{Q}(x, y)=ax^2-bxy+cy^2$ is also reduced. Moreover, $a(n, Q)=a(n, \bar{Q})$.

(iii) The integers $a$ and $c$ are the smallest and second smallest non-zero integers represented by $Q(x, y)$, respectively.

\end{lemma}
\begin{proof}
(i) Suppose $Q(x,y)=ax^2+bxy+cy^2$ is a reduced form that is not principal. Hence, $|b|\leq a \leq c$, and $a>1$,
by Lemma \ref{leastpositive}. 
We will prove $a\neq |b|$ and $a\neq c$ by contradiction. 
Suppose $a=|b|$, then 
$a^2-4ac=-D.$ 
Since $a$ divides the left-hand side, $a$ divides $D$. Therefore,  $a=D$, since $a>1$ and, by Lemma \ref{odd}, $D$ is prime. This contradicts the upper bound ${{\sqrt{D/3}}{}}$ for $a$ (see \cite[Equation (2.12)]{DC}) for a reduced form. Hence, $a\neq |b|$. 

Next, we prove $a\neq c$. Suppose $a=c$. Then, from
  $b^2-4a^2=-D,$
we get
  $D=(2a-b)(2a+b),$
contradicting the fact that $D$ is prime (note that $|b|\leq a$ and $a>1$, hence $2a-b\neq 1$ and $2a+b\neq 1$).
 
(ii) From part (i), it is evident that if $Q(x, y)$ is reduced, then $\bar{Q}(x, y)$ is also reduced. To see the relation between representations by these forms, note that
for integers $x_0$ and $y_0$, we have  $$Q(x_0,y_0)={Q}(-x_0,-y_0)=\bar{Q}(x_0,-y_0)=\bar{Q}(-x_0,y_0).$$

(iii) Observe that
$$Q(x,y)\geq \left( a-|b|+c\right)\min(x^2,y^2).$$
By part (i), $a-|b|>0$, hence  
$$Q(x,y)> c\min(x^2,y^2)$$
when $x\neq 0$ and $y\neq 0$. Therefore, $Q(x, y)> c$ for $x\neq0$ and $y\neq0$. The result follows by noting that the value $a$ is achieved at $(1, 0)$,  the value $c$ is taken at $(0, 1)$, and that $a<c$ by part (i). \end{proof}
\section{Proof of Theorem \ref{main-1}}
\label{Section 3}
\begin{proof}
For $\Re(s)>1$, let 
\begin{equation}
\label{eq2}
\zeta(K,s) 
=\prod_{p}\left(1-\frac{1}{p^{s}}\right)^{-1}  \prod_{p}\left(1-\frac{\Leg{-D}{p}}{p^{s}}\right)^{-1}
\end{equation}
be the Dedekind zeta function of $K=\mathbb{Q}(\sqrt{-D})$.
Let $Q_0(x, y), Q_1(x, y), \bar{Q}_1(x, y), \ldots, Q_k(x, y),$ and $\bar{Q}_k(x, y)$ be reduced forms associated with the binary quadratic forms of discriminant $-D$, where $Q_0(x, y)$ represents the principal form and $\bar{Q}_i(x, y)$ is the conjugate form corresponding to ${Q}_i(x, y)$. 
By employing \cite[Lemma 27]{stark}, we have, for $\Re(s)>1$,
\begin{equation}
\label{eq1}
    \zeta(K,s)
    =\frac{1}{w} \sum_{n}\left(\frac{a(n,Q_0)+a(n, Q_1)+a(n,\bar{Q}_1)+\cdots+ a(n, Q_k)+a(n,\bar{Q}_k) }{n^s}\right),
\end{equation}
where $w$ is the number of roots of unity in $K$.
Then, by equating the coefficients of $1/n^s$ in \eqref{eq2} and \eqref{eq1}, and employing Lemma \eqref{conjugate} (ii), we get
\begin{equation}
\label{d1}
a\left(n, Q_{0}\right) + 2  \sum_{r=1}^{k} a\left(n, Q_{r} \right)
 = w\sum_{d\mid n}  \Leg{-D}{d}.\footnote{This identity is a special instance of the Siegel-Weil mass formula. For an alternate proof see \cite[Chapter 8, Satz 3]{Z}.}
 \end{equation}
 
Let  $\Theta_Q(z)=\sum_{a,b\in \mathbb{Z}}^{}q^{Q(a,b)}$ be the theta function associated to the reduced form $Q(x, y)$ of discriminant $-D$.
%
From \cite[Theorem 10.1]{XWDP} follows that
$\Theta_Q(z)$ is a modular form of weight 1, level $D$, and character $\Leg{-D}{.}$. 
It is known that if $Q_1(x, y)$ and $Q_2(x, y)$ are in the same genus (i.e, represent the same values in $(\mathbb{Z}/D\mathbb{Z})^\times$), then   $\Theta_{Q_1}(z)- \Theta_{Q_2}(z)$  is a cusp form (see \cite[p. 365]{XWDP} or \cite{W} for a proof).  
From \cite[Corollary 3.14 (ii)]{DC} we know that all genera of forms (i.e., the collection of reduced forms in the same genus)  of discriminant $-D$ consist of the same number of classes, and  the number of genera of forms is a power of two. Hence, as $h(-D)$ is odd, any two reduced forms of discriminant $-D$ are in the same genus.
Thus, for $1\leq r\leq k$,
$$F_{D, r}(z)=\frac{1}{2}(\Theta_{Q_0}(z)-\Theta_{Q_r}(z))=\sum_{n=1}^{\infty} t_r(n) q^n,$$
with $t_r(1)=1$, is a cusp form of weight 1, level $D$ and character $\Leg{-D}{.}$. Therefore, 
\begin{equation}
\label{d2}
a(n, Q_0)-a(n, Q_r)=2t_r(n).
\end{equation}

The result follows from the two displayed identities \eqref{d1} and \eqref{d2}.
\end{proof}
\section{Infinite product representation of periodic holomorphic functions on $\mathbb{H}$}\label{Serre}
We start by an assertion due to Eholzer and Skoruppa on the infinite product expansions of periodic holomorphic functions on the upper half-plane. 
\begin{lemma}{\em (\cite[Section 2]{ES})}\label{kohnen}
Let $f$ be a holomorphic function on the upper half plane such that $f(z+1)=f(z)$.
Let $f(z)=\sum_{n=0}^\infty a_f(n)q^n$ be the Fourier expansion of $f$. In addition, assume that $a_f(0)=1$. 
Then there exists a unique sequence of complex numbers $c(n)$ such that 
\begin{equation}\label{eholzer}
f(z)=\prod_{n=1}^\infty (1-q^n)^{c(n)}    
\end{equation}
for sufficiently small $|q|$. Moreover, the $c(n)$'s are integral if $f$ has integral Fourier coefficients. 
\end{lemma}
 We will use \lemref{kohnen} to 
prove that for none of the values of $(D,r)$ other than (23,1), $F_{D, r}(z)$ can be written as an eta quotient. The important fact to observe here is that the exponents $c(n)$ in \eqref{eholzer} are not always bounded. 
\begin{proposition}\label{unboundedlemma}
Let $f$ be as in Lemma \ref{kohnen}. If $f$ has a zero on the upper half plane,
then the complex numbers $c(n)$ in the expansion \eqref{eholzer} 
are unbounded. More generally, for any $\alpha\in \mathbb{R}$,
$$\frac{1}{n^{\alpha}}{\displaystyle{\sum_{d\mid n}} d c(d)}$$
is not bounded.

\end{proposition}
\begin{proof}
By \lemref{kohnen},  $f(z)$ has an expansion in the form \eqref{eholzer}
for some complex numbers $c(n).$ We are given that $f$ has a zero in $\mathbb{H}$. Since $f$ is holomorphic and not identically zero on $\mathbb{H}$ (since $a_f(0)=1$), the zeros of $f$ are discrete and hence there exists a neighborhood $N(i\infty)$ of $i\infty$ such that $f$ has no zeros in $N(i\infty)\cap\mathbb{H}$. Therefore, there exists a $T>0$ such that for every $z$ with $\mathrm{Im}(z)>T$, $f(z)\neq 0$ on $\mathbb{H}$. Let $z_0$ be a zero with the largest imaginary part. Set $F(q):=f(z)$. Then $q_0=e^{2\pi i z_0}$ is a point of singularity for   
 $$L(q)=\log\left(F(q)\right).$$  
Moreover, $L(q)$ has a power series expansion on the disc centred at zero with a radius of convergence $R\leq |q_0|= e^{-2\pi \mathrm{Im}(z_0)}$.
By \eqref{eholzer}, we have, for $|q|\leq R$,
\begin{equation*}
\begin{split}
    L(q) 
    & = \sum_{n=1}^\infty\left(-\frac{1}{n} \sum_{d|n}dc(d)\right)q^n.
\end{split}
\end{equation*}
Hence,
\begin{equation}
\label{radius}
 \mathrm{limsup}\left(\sqrt[n]{\left| \frac{1}{n}\sum_{d|n}dc(d) \right|}\right)\geq e^{2\pi \mathrm{Im}(z_0)}.
 \end{equation}
 Now if $c(n)$'s are bounded, $\sum_{d|n}dc(d)
 =O\left( n^{1+\epsilon} \right)$ for any $\epsilon>0$. Thus, by \eqref{radius},
we have $e^{2\pi \mathrm{Im}(z_0)}\leq 1$, which is a contradiction since $\mathrm{Im}(z_0)>0$. Therefore, $c(n)$'s must be unbounded. The general statement follows similarly.
\end{proof}

\section{Valence formula and proof of Theorem \ref{cn}}
\label{Section 5}
We denote the order of vanishing  of a modular form $f$ at a point $z\in{\mathbb{H}}^{*}=\mathbb{H}\cup\mathbb{Q}\cup\{i\infty\}$ by $\nu_z(f)$. 

\begin{lemma}\em (\cite[Theorem 5.6.11]{HC})\label{valencef} Let $\Gamma$ be a subgroup of ${\rm SL}_2(\mathbb{Z})$ of finite index and let $f\neq 0$ be a modular form of weight $k$ for $\Gamma$. We have
$$
\sum_{z \in \Gamma \backslash {\mathbb{H}}^{*}} \frac{\nu_{z}(f)}{e_{z}}=\frac{k}{12}[\bar{{\rm SL}_2(\mathbb{Z})}:\bar{\Gamma}],
$$
where $e_z=2$ or $3$ if $z$ is ${\rm SL}_2(\mathbb{Z})$ equivalent to $i$ or $e^{2\pi i/3}$ respectively, and $e_z=1$ otherwise. Here $\bar{{\rm SL}_2(\mathbb{Z})}={\rm SL}_2(\mathbb{Z})/\{{\pm I\}}$ and $\bar{\Gamma}=\Gamma/(\Gamma\cap\{{\pm I\}})$.
\end{lemma}
The above theorem works for modular forms on a finite index subgroup with a trivial character. Although we have that $F_{D, r}(z)\in M_1\left(\Gamma_0(D),\Leg{-D}{.}\right)$, we can still use the valence formula on $F_{D, r}(z)$ by considering it as a modular form on $\Gamma_1(N)$. To do so, we will have to look at the order of vanishing of $F_{D, r}(z)$ at all the cusps of $\Gamma_1(D)$. For $D$ prime, $\Gamma_1(D)$ has $D-1$ inequivalent cusps (See \cite[Page 102]{FD}). The next proposition describes the order of vanishing of $F_{D, r}(z)$ at these cusps.

\begin{proposition}\label{cuspall}
Let the integer $-D<0$ be a fundamental discriminant and assume that $h(-D)=2k+1$ for an integer $k\geq 1$. 
Let $F_{D, r}(z)$ be as in Theorem \ref{main-1}.
Then, the order of vanishing of $F_{D, r}(z)$ at every cusp of $\Gamma_1(D)$ is 1.
\end{proposition}
\begin{proof}
Since $D$ is prime, the cusps $i\infty$ and $1/1$ are two inequivalent cusps. 
We first show that the order of vanishing $\nu_z(F_{D, r})$ in $\Gamma_1(D)$ at the cusps $z=i\infty$ and $z=1/1$ are $1$. 
At the cusp $i\infty$, by employing \lemref{leastpositive} and \lemref{conjugate} (iii), the Fourier expansions of the theta functions are given by 
\begin{equation*}
\Theta_{Q_0}(z)=\sum_{n=0}^\infty a(n,Q_0)q^n=1+2q+O\left( q^2 \right),    
\end{equation*}
and 
\begin{equation*}
\Theta_{Q_r}(z)=\sum_{n=0}^\infty a(n,Q_r)q^n=1+O\left( q^2 \right).
\end{equation*}
Hence, 
\begin{equation}\label{oov1}
F_{D, r}(z)=q+O\left( q^2 \right).
\end{equation}
Next, we find out the Fourier expansions of the theta functions at the cusp 1/1. Recall that, by Lemma \ref{odd}, $D\equiv 3 \pmod{4}$.
For $Q_0(x, y)=x^2+xy+\frac{1+D}{4} y^2$, using the formula for the Fourier expansion of theta functions in the proof of \cite[Theorem 10.1]{XWDP}, with
$\rho=\begin{psmallmatrix}1 & 0\\1 & 1\end{psmallmatrix}$,
we get
\begin{equation}
\label{above}
    (z+1)^{-1}\Theta_{Q_0}(\rho z)  = \frac{-i}{\sqrt{D}}\sum_{(x,y)\in\mathbb{Z}^2}q^{\frac{1}{D}\left(\frac{1+D}{4}x^2-xy+y^2\right)}e^{\frac{2\pi i}{D}\left(\frac{1+D}{4}x^2-xy+y^2 \right)}. 
    \end{equation}
(Note that $\rho(i\infty)=1/1$.) For a complex function $f$ defined in $\mathbb{H}$, and $\rho=\begin{psmallmatrix}
  a & b\\ 
  c & d
\end{psmallmatrix}\in {\rm SL}_2(\mathbb{Z})$, we set
$$f[\rho]_k(z)=(cz+d)^{-k} f(\rho z).$$
Using this notation and observing that the form $\frac{1+D}{4}x^2-xy+y^2$ in 
\eqref{above} is properly equivalent to $Q_0(x, y)$, we have
    \begin{equation*}
  \Theta_{Q_0}[\rho]_1(z) = \frac{-i}{\sqrt{D}}\sum_{(x,y)\in\mathbb{Z}^2}q^{\frac{Q_{0}(x,y)}{D}}e^{\frac{2\pi iQ_{0}(x,y)}{D}}.
\end{equation*}
Hence,
\begin{equation}\label{theta11}
    \Theta_{Q_0}[\rho]_1(z)=\frac{-i}{\sqrt{D}}\left( 1+c_1q^{\frac{1}{D}}+\cdots \right),
\end{equation}
where $c_1$ is twice a root of unity and thus is nonzero.
Similarly for the non-principal reduced quadratic form $Q_r(x,y)=a_0x^2+b_0xy+c_0y^2$, we have 
\begin{equation}
    \label{theta12}\Theta_{Q_r}[\rho]_1(z)=\frac{-i}{\sqrt{D}}\left( 1+c_2q^{\frac{a_0}{D}}+\cdots \right),    
\end{equation}
where $c_2\neq 0$ and $a_0>1$. Here, we used Lemmas \ref{leastpositive} and \ref{conjugate} (iii) and the fact that $Q_r(x, y)$ is non-principal.
Now, from \eqref{theta11} and \eqref{theta12}, 
\begin{equation}\label{oov2}
    F_{D, r}[\rho]_1(z)=\frac{1}{2}\left(\Theta_{Q_0}[\rho]_1(z)-\Theta_{Q_r}[\rho]_1(z)\right)=\frac{-ic_1}{2\sqrt{D}}q_D + O\left( q_D^2 \right),
\end{equation}
where $q_D=q^\frac{1}{D}$. 
Using \cite[Proposition 6.3.20]{HC}, the width of the cusps $i\infty$ and $1/1$ in $\Gamma_1(D)$ are equal to 1 and $D$ respectively. Hence, by \eqref{oov1} and \eqref{oov2},
the order of vanishings of $F_{D, r}(z)$ at these cusps of $\Gamma_1(D)$ are 1.

The Group $\Gamma_1(D)$ has $D-1$ inequivalent cusps. 
By \cite[Proposition 6.3.19]{HC},
a set of representatives for the inequivalent cusps can be given by $S_1\cup S_2$, where 
$$S_1=\left\{ i\infty ,\frac{2}{D},\frac{3}{D},\cdots,\frac{(D-1)/2}{D} \right\}
~~~{\rm and}~~~
S_2=\left\{\frac{1}{1},\frac{1}{2},\frac{1}{3},\cdots,\frac{1}{(D-1)/2}\right\}.$$
Observe that any two cusps in $S_1$ or $S_2$ are $\Gamma_0(D)$ equivalent. 
Next, by \cite[Proposition 6.3.20]{HC}, we have that any cusp in $S_1$ has width one and any cusp in $S_2$ has width $D$. Hence, any two $\Gamma_0(D)$ equivalent cusps have the same width in $\Gamma_1(D)$. Suppose $x_1$ and $x_2$ are two inequivaent cusps of $\Gamma_1(D)$ that are $\Gamma_0(D)$ equivalent. Let $w$ be their width. Then, as shown in the proof of \cite[Proposition 16]{Kob}, the smallest exponents in the $q_w$ expansion of $F_{D, r}(z)$ at the cusps $x_1$ and $x_2$ are the same. Since the widths of $x_1$ and $x_2$ are the same and equal to $w$, the smallest exponent with non-zero coefficients in $q_w$ expansion of $F_{D, r}(z)$ at the cusps $x_1$ and $x_2$ are the same. Hence $F_{D, r}(z)$ has the same order of vanishing in $\Gamma_1(D)$ at the cusps $x_1$ and $x_2$. Now since the order of vanishing of $i\infty$ and $1/1$ in $\Gamma_1(D)$ is equal to 1, and every cusp of $\Gamma_1(D)$ is $\Gamma_0(D)$ equivalent to either $i\infty$ or $1/1$, we are done. 
\end{proof}

We are ready to prove Theorem \ref{cn}.

\begin{proof}[Proof of Theorem \ref{cn}]
Considering $F_{D, r}(z)$ as a modular form on $\Gamma_1(D)$, by \lemref{valencef}, we have  $$
\sum_{z \in \Gamma \backslash {\mathbb{H}}^{*}} \frac{\nu_{z}(F_{D, r})}{e_{z}}=\frac{1}{12}[\bar{{\rm SL}_2(\mathbb{Z})}:\bar{\Gamma_1(D)}].
$$
By \cite[Corollary 6.2.13]{HC}, we have
$$
[\bar{{\rm SL}_2(\mathbb{Z})}:\bar{\Gamma_1(D)}]=\frac{D^2-1}{2}.
$$
Hence, we have
$$
\sum_{z \in \Gamma_1(D) \backslash {\mathbb{H}}^{*}} \frac{\nu_{z}(F_{D, r})}{e_{z}}=\frac{D^2-1}{24}.
$$
From \lemref{cuspall}, $\nu_z(F_{D, r})=1$ for each of the $D-1$ cusps of $\Gamma_1(D)$. Hence, 
$$
\frac{1}{D-1}\sum_{z \in \Gamma_1(D) \backslash \mathbb{H}} \frac{\nu_{z}(F_{D, r})}{e_{z}}=\frac{D-23}{24}.
$$
Therefore, for all values of $D$ except 23, the right-hand side will be positive, and the formula above will guarantee the existence of a zero in $\mathbb{H}$. Now applying \propref{unboundedlemma} on $F_{D, r}(z)/q$ completes the proof.
 \end{proof}

\section{Schoeneberg pairs and proof of Theorem \ref{Pair}}
\label{Section 6}

\begin{proposition}
\label{eta-23}
Let $D$ be a prime, and let $f(z) \in S_1\left(\Gamma_0(D), \Leg{-D}{\cdot}\right)$ be an eta quotient. Then, $D\equiv 23\pmod{24}$, and 
\begin{equation}
\label{eta-display}
f(z)=\eta(z) \eta(Dz)= q^{\frac{D+1}{24}}\prod_{n=1}^\infty (1-q^n)^{c(n)},
~~{\rm where}~~
c(n)= \begin{cases}
2& {\rm if}~ n\equiv 0~({\rm mod}~D),\\
1& {\rm Otherwise.}
\end{cases}
\end{equation}
\end{proposition}

\begin{proof}
Let $f(z)=\prod_{\delta\mid N} \eta(\delta z)^{r_\delta}\in S_1\left(\Gamma_0(D), \Leg{-D}{\cdot}\right)$ be an {eta quotient}. Since $f(z)$ satisfies the transformation property of a modular form of level $D$, and $D$ is prime, then the largest subgroup of  ${\rm SL}_2(\mathbb{Z})$ for which $f(z)$ transforms like a modular form is either $\Gamma_0(1)={\rm SL}_2(\mathbb{Z})$ or $\Gamma_0(D)$.
Hence, by \cite[p. 880]{B}, we conclude that ${\rm lcm}\{\delta;~\delta\mid N~{\rm and}~r_\delta\neq 0\}=1$ or $D$. Since $D$ is prime, we have 
$$f(z)=\eta^i(z)\eta^j(Dz),$$ where $i$ and $j$ are integers.  

Now since $\eta(z)$ is a modular form of weight $1/2$, $f(z)= \eta^i(z)\eta^j(Dz)$ transforms like a modular form of weight $(i+j)/2$. On the other hand  we assumed that $f(z)$ has weight  1, hence
\begin{equation}\label{eta11}
i+j=2.
\end{equation}
For $f(z)$ to be holomorphic at the cusps $i\infty$ and $1/1$, by \cite[Theorem 1.65]{O}, we must have
\begin{equation}\label{eta12}
\frac{i+Dj}{24}\in \mathbb{N}\cup\{0\}
\end{equation}
and
\begin{equation}\label{eta13}
\frac{Di+j}{24}\in \mathbb{N}\cup\{0\}.
\end{equation}
By \eqref{eta11}, \eqref{eta12} and \eqref{eta13}, we have that
\begin{equation}\label{eta14}
D\equiv11\pmod{12}.
\end{equation}
We subtract \eqref{eta13} from \eqref{eta12} to get
\begin{equation}\label{eta15}
(1-D)(i-j)\equiv 0\pmod{24}.
\end{equation}
From \eqref{eta14} and \eqref{eta15}, we come to the conclusion that $i=6\ell+1$ and $j=-6\ell+1$, for an integer $\ell$. Thus, $f(z)=\eta^{6\ell+1}(z) \eta^{-6\ell+1}(Dz)$. Since $f(z)$ is a cusp form, it vanishes at $\infty$ and $1/1$, so \eqref{eta12} and  \eqref{eta13} can be written as 
\begin{equation}
\label{b0}
\frac{(D+1)+6\ell (1-D)}{24}\in \mathbb{N},~~~{\rm and}~~~ \frac{(D+1)-6\ell (1-D)}{24}\in \mathbb{N}.
\end{equation}
From the positivity of the numerators of the fractions in \eqref{b0}, we conclude that 
\begin{equation}
\label{b1}
|6\ell| < \frac{D+1}{D-1}.
\end{equation}
Since $\ell=0$ is the only solution of \eqref{b1}, from $i=6\ell+1$, $j=-6\ell+1$, and \eqref{b0}, we conclude that $i=j=1$ and $D\equiv 23$ (mod 24). Note that $\eta(z) \eta(Dz)\in  S_1\left(\Gamma_0(D), \Leg{-D}{\cdot}\right)$ by \cite[Theorem 1.65]{O}.
\end{proof}

We are now ready to prove our final result.
\begin{proof}[Proof of Theorem \ref{Pair}]
First assume $E(z)=F_{D, s, r}(z)$, where $E(z)$ is an eta-quotient of prime level $D$ and $F_{D, s, r}(z)$ is half the difference of two theta functions of  prime level $D$.  Since $D$ is a prime and $-D$ is a discriminant, then $D\equiv 3$ (mod 4). Hence, $h(-D)$ is odd and by Theorem \ref{main-1}, $F_{D, s, r}(z)=F_{D, r}(z)-F_{D, s}(z)$ is a cusp form of weight $1$, level $D$, and character $\Leg{-D}{.}$. Thus, by Proposition \ref{eta-23}, $D\equiv 23$ (mod 24) and $E(z)=\eta(z) \eta(Dz)$.

Conversely assume that $D\equiv 23$ (mod 24) is a prime. 
Then $-D$ is a fundamental discriminant and $h(-D)$ is odd.
Let $Q=(a,b,c)$ represent the reduced form $Q(x, y)=ax^2+bxy+cy^2$. Suppose that  $Q_{s}=(a,b,c)$ and $Q_{r}=(a^\prime,b^\prime,c^\prime)$ be a pair of distinct reduced forms of Discriminant $-D$ for which $F_{D, s, r}(z)$ is an eta quotient. Then, by  Theorem \ref{main-1} and Proposition \ref{eta-23}, $F_{D, s, r}(z)=\eta(z)\eta(Dz)$. By  \cite[Satz 1, p. 179]{Sch} such pair $(Q_s, Q_r)$ exists. We claim that up to proper equivalence and conjugation it is the Schoeneberg pair. We consider two cases.

Case 1: Assume that $a<a^\prime$. Without loss of generality, further assume that $b>0$ and $b^\prime>0$. Then since, by Lemma \ref{conjugate} (iii), $a$ is the smallest positive integer represented by $Q_{s}$ then the least exponent of $q$ in the $q$-expansion of $F_{D, s, r}(z)$ at $i\infty$ is $a$. Thus, as $F_{D, s, r}(z)$ is an eta-quotient, by Proposition \ref{eta-23}, we have $a=(D+1)/24$. Since $Q_{s}$ is a reduced form, we have $|b|\leq a\leq c$. Hence, 
\begin{equation}
\label{inequality}
\frac{(D+1)^2}{24^2}\geq b^2=4ac-D\geq \frac{4(D+1)^2}{24^2}-D.
\end{equation}
Here, we used that $a=(D+1)/24$ and $c\geq a$.
The inequality \eqref{inequality} implies $1\leq D\leq 189.99$. Using SageMath \cite{sage} we produce the list of reduced forms of prime discriminants $-D$ with $D\equiv 23\pmod{24}$ with $1\leq D\leq 189.99$. By examining them, we conclude that the Schoeneberg pairs listed in Table \ref{t1-intro} are the only pairs $(Q_s, Q_r)$ for which $F_{D, s, r}(z)$ is an eta quotient. 

Case 2: Assume that $a=a^\prime$. Hence, $Q_{s}=(a,b,c)$ and $Q_{r}=(a,b^\prime,c^\prime)$. Without loss of generality, further assume that $b>0$, $b^\prime>0$, and $c<c^\prime$. Since $c<c^\prime$, the least exponent of $q$ in $F_{D, s, r}(z)$ has to be equal to $c$ because, by Lemma \ref{conjugate} (iii), $c$ is the least integer represented by $Q_{s}$ that is not represented by $Q_{r}$. Since $F_{D, s, r}(z)$ is an eta-quotient, by \eqref{eta-display}, we have $c=(D+1)/24$.  Hence,
\begin{equation}
\label{inequality2}
4ac-D=b^2\leq a^2\leq \frac{D}{3}.
\end{equation}
(For the last inequality, see \cite[Equation (2.12)]{DC}.) For $a\geq 9$, the inequalty \eqref{inequality2} implies that $c<(D+1)/24$,
which is a contradiction. Hence, such a pair of forms does not exist if $a\geq 9$.

When $a=8$ and $c=(D+1)/24$, we have
$$b^2-\frac{4(D+1)}{3}=-D$$ which implies $b^2=(D+4)/3$. Since $c\geq a$, we have $(D+1)/24\geq 8$. Rearranging, we get $b^2=(D+4)/3\geq 65$, which is impossible as the form is reduced ($b^2\leq a^2=8^2=64$). 

When $a=7$ and $c=(D+1)/24$, we have
$$b^2-\frac{7(D+1)}{6}=-D$$ which implies $b^2=(D+7)/6$. Since $c\geq a$, we have $(D+1)/24\geq 7$. Rearranging, we get $(D+7)/6\geq 29$. Hence, $29\leq b^2\leq a^2=7^2=49$, which implies $b=6$ or $7$. If $b=7$, then $D=287$. This is not possible as $287$ is not prime. The case $b=6$ does not result in an integral value for $c$. 

Next assume that $a=6$ and $c=(D+1)/24$. Then $-D=b^2-4ac$
implies $b=1$ and therefore $Q_s=(6, 1,  (D+1)/24)$. Now 
$${b^\prime}^2-24c^\prime=-D.$$ Hence, $c^\prime=(D+{b^\prime}^2)/24.$ Using the fact that $c^\prime$ is an integer and $b^\prime\leq a^\prime=a=6$, we get $b^\prime=1$ or $5$. (Note that since $D$ is prime, then $b^\prime=2, 3, 4, 6$ is not possible.) Now if $b^\prime=1$, then $Q_s$ and $Q_r$ are not distinct, which contradicts our assumption. The remaining case $b^\prime=5$ results in Schoeneberg
pairs for all values of prime $D\equiv 23\pmod{24}$ such that $D\geq143$. 

Finally, for $1\leq a\leq 5$, and $c=(D+1)/24$, we have $$b^2\leq \frac{20(D+1)}{24}-D,$$ which only gives solutions for $b$ if $D\leq 5$. The only possible cases ($D=3,4$) do not satisfy $D\equiv 23\pmod{24}$.\end{proof}

\newpage
\section{Class numbers one, three, and five}
\label{appendix}
In this section, we illustrate explicitly the assertions of Theorem \ref{main-1} for primitive positive definite quadratic forms (or equivalently for imaginary quadratic fields) of class number one, three, and five. 
All computations are done with SageMath \cite{sage}.

\begin{itemize}
\item {\sl Class number 1}:
{Table \ref{B} records the discriminants $-D$ with class number one, the reduced form $Q(x, y)$, and number of roots of unity $w$.}

\begin{table}[ht]
\scriptsize
\centering
\begin{tabular}[t]{|c| c| c|}  \hline
 $-D$ & $Q(x, y)$ & $w$ \\ 
 \hline
 $-3$ & $x^2+xy+y^2$ & 6  \\ 
 \hline
 $-4$ & $x^2 + y^2$ & 4 \\
 \hline
 $-7$ & $x^2 + xy + 2y^2$ &2 \\
 \hline
 $-8$ & $x^2 + 2y^2$  &2  \\
 \hline
 $-11$ & $x^2 + xy + 3y^2$ &2  \\ 
 \hline
 $-19$ & $x^2 + xy + 5y^2$ &2   \\ 
 \hline
 $-43$ & $x^2 + xy + 11y^2$ &2   \\ 
 \hline
 $-67$ & $x^2 + xy + 17y^2$ &2   \\
 \hline
 $-163$ & $x^2 + xy + 41y^2$ &2 \\ 
 \hline

\end{tabular} \caption{Class number 1} \label{B}
\end{table}

\item {\sl Class number 3 (a)}: {Table \ref{A} records the discriminants $-D$ with class number three and  \\$\mathrm{dim}(S_1(\Gamma_0(D), \Leg{-D}{.})=1$, and reduced forms $Q_0(x, y)$, $Q_1(x, y)$. The cusp form $F_{D,1}(z)$ is given as  the $q$-expasion of the element of the basis of $S_1(\Gamma_0(D),\Leg{-D}{.})$ produced by SageMath.}

 \begin{table}[ht]
\scriptsize
\centering
\begin{tabular}[t]{|c| c| c| c|}  \hline
 $-D$ & $Q_0(x, y)$ & $Q_1(x, y)$  & $F_{D,1}(z)=\displaystyle{\sum_{n=1}^{\infty} t_1(n) q^n}$\\ 
 \hline
 $-23$ & $x^2 + xy + 6y^2$ & $2x^2 + xy + 3y^2$ & $q-q^2-q^3+q^6+q^8-q^{13}-q^{16}+q^{23}-q^{24} +q^{25}+O(q^{26})$\\ 
 \hline
 $-31$ & $x^2 + xy + 8y^2$ & $2x^2 + xy + 4y^2$  & $q -  q^{2} -  q^{5} -  q^{7} +  q^{8} + q^9 + q^{10} +q^{14} -q^{16} -q^{18} + O(q^{19})$ \\
 \hline
 $-59$ & $x^2 + xy + 15y^2$ & $3x^2 + xy + 5y^2$  & $q -  q^{3} +  q^{4} -  q^{5} -  q^{7} -  q^{12} +  q^{15} +  q^{16} + 2 q^{17} -  q^{19} +  O(q^{20})$\\
 \hline
 $-83$ & $x^2 + xy + 21y^2$ & $3x^2 + xy + 7y^2$  & $q -  q^{3} +  q^{4} -  q^{7} -  q^{11} -  q^{12} +  q^{16} -  q^{17} +  q^{21} + 2 q^{23} +  O(q^{25})$\\
 \hline
 $-107$ & $x^2 + xy + 27y^2$ & $3x^2 + xy + 9y^2$  & $q -  q^{3} +  q^{4} -  q^{11} -  q^{12} -  q^{13} +  q^{16} -  q^{19} -  q^{23} +  q^{25} +  O(q^{27})$\\ 
 \hline
 $-139$ & $x^2 + xy + 35y^2$ & $5x^2 + xy + 7y^2$ & $q +  q^{4} -  q^{5} -  q^{7} +  q^{9} -  q^{11} -  q^{13} +  q^{16} -  q^{20} -  q^{28}  +  O(q^{29})$\\ 
 \hline
 $-211$ & $x^2 + xy + 53y^2$ & $5x^2 + 3xy + 11y^2$  & $ q +  q^{4} -  q^{5} +  q^{9} -  q^{11} -  q^{13} +  q^{16} -  q^{19} -  q^{20} +  q^{36} +  O(q^{37})$\\ 
 \hline
 $-307$ & $x^2 + xy + 77y^2$ & $7x^2 + xy + 11y^2$ & $q +  q^{4} -  q^{7} +  q^{9} -  q^{11} +  q^{16} -  q^{17} -  q^{19} +  q^{25} -  q^{28} +  O(q^{36} )$\\ 
 \hline
 $-379$ & $x^2 + xy + 95y^2$ & $5x^2 + xy + 19y^2$  & $q +  q^{4} -  q^{5} +  q^{9} +  q^{16} -  q^{19} -  q^{20} -  q^{23} +  q^{36} -  q^{37} + O(q^{41} )$\\
 \hline
$-499$ & $x^2 + xy + 125y^2$ & $5x^2 + xy + 25y^2$  & $ q +  q^{4} -  q^{5} +  q^{9} +  q^{16} -  q^{20} -  q^{29} -  q^{31} +  q^{36} -  q^{43} +  O(q^{45})$\\ 
 \hline
 $-547$ & $x^2 + xy + 137y^2$ & $11x^2 + 5xy + 13y^2$  & $q +  q^{4} +  q^{9} -  q^{11} -  q^{13} +  q^{16} -  q^{19} +  q^{25} -  q^{29} +  q^{36} +  O(q^{44})$\\ 
 \hline
$-883$ & $x^2 + xy + 221y^2$ & $13x^2 + xy + 17y^2$ & $q +  q^{4} +  q^{9} -  q^{13} +  q^{16} -  q^{17} +  q^{25} -  q^{29} -  q^{31} +  q^{36} +  O(q^{49} )$\\ 
 \hline
$-907$ & $x^2 + xy + 227y^2$ & $13x^2 + 9xy + 19y^2$ & $q +  q^{4} +  q^{9} -  q^{13} +  q^{16} -  q^{19} -  q^{23} +  q^{25} +  q^{36} -  q^{41} +  O(q^{49})$\\ \hline  
\end{tabular} \caption{Class number 3 (a)} \label{A} \end{table}


\item  {\sl Class number 3 (b):} 
{Table \ref{AA} records discriminants $-D$ with class number three and\\  
$\mathrm{dim}(S_1(\Gamma_0(D),\Leg{-D}{.})=3$, and reduced forms $Q_0(x, y)$, $Q_1(x, y)$. The cusp form $F_{D,1}(z)$ is $\mathcal{S}_1+\mathcal{S}_3$ where $\mathcal{S}_1,\mathcal{S}_2,\mathcal{S}_3$ are elements of the basis of $S_1(\Gamma_0(D),\Leg{-D}{.})$ produced by SageMath.}
 
 \begin{table}[ht]
\scriptsize
\centering
\begin{tabular}[t]{|c| c| c| c|}  \hline
 $-D$ & $Q_0(x, y)$ & $Q_1(x, y)$  & $F_{D,1}(z)=\displaystyle{\sum_{n=1}^{\infty} t_1(n) q^n}$\\ 
 \hline
 $-283$ & $x^2 + xy + 71y^2$ & $7x^2 + 5xy + 11y^2$  & $q+q^4-q^7+q^9-q^{11}-q^{13}+q^{16}-q^{23}+q^{25}-q^{28}+O(q^{29})$\\
 \hline
 $-331$ & $x^2 + xy + 83y^2$ & $5x^2 + 3xy + 17y^2$  & $q+q^4-q^5+q^9+q^{16}-q^{17}-q^{19}-q^{20}-q^{31}+q^{36}+O(q^{43})$\\ 
 \hline
$-643$ & $x^2 + xy + 161y^2$ & $7x^2 + xy + 23y^2$  & $q+q^4-q^7+q^9+q^{16}-q^{23}+q^{25}-q^{28}-q^{29}-q^{31}+O(q^{36})$\\
 \hline
\end{tabular} \caption{Class number 3 (b)} \label{AA}
\end{table}

\newpage
\item {\sl Class number 5:} {Table \ref{C} records discriminants $-D$ with class number five, \\$\mathrm{dim}(S_1(\Gamma_0(D),\Leg{-D}{.})=2$, and
reduced forms $Q_0(x, y)$, $Q_1(x, y)$, $Q_2(x, y)$. The cusp form $F_{D, 1}(z)$  is $\mathcal{S}_1-\mathcal{S}_2$ and $F_{D,2}(z)$ is $\mathcal{S}_1$, where    
$\mathcal{S}_1$ and $\mathcal{S}_2$ are elements of the basis of $S_1(\Gamma_0(D),\Leg{-D}{.})$  produced by SageMath.}

 \begin{table}[ht]
\scriptsize
\centering
\begin{tabular}[t]{|c|l|l|l|l|}  \hline
 $-D$ & $Q_0(x, y)$ & $Q_1(x, y)$ & $Q_2(x, y)$  & $F_{D,1}(z)=\displaystyle{\sum_{n=1}^{\infty} t_1(n) q^n},~ F_{D,2}(z)=\displaystyle{\sum_{n=1}^{\infty} t_2(n) q^n}$  \\ 
 \hline
 $-47$ & $x^2 + xy + 12y^2$ & $2x^2 + xy + 6y^2$ & $3x^2 + xy + 4y^2$ & \begin{tabular}{@{}c@{}}$q-q^2+q^4-q^6-q^7-q^8+q^{12}+q^{18}+O(q^{21})$ \\ $q-q^3-q^6-q^8+q^9+q^{12}+q^{14}-q^{17}+O(q^{18})$ \end{tabular}  \\ 
 \hline
 $-79$ & $x^2 + xy + 20y^2$ & $2x^2 + xy + 10y^2$  & $4x^2 + xy + 5y^2$ & \begin{tabular}{@{}c@{}}$q-q^2+q^4-q^8+q^9-q^{10}-q^{13}-q^{18}+O(q^{20})$ \\ $q - q^5 - q^8 + q^9 - q^{10} - q^{19} + q^{20} + q^{22} + O(q^{23})$ \end{tabular} \\
 \hline
 $-103$ &  $x^2 + xy + 26y^2$ & $2x^2 + xy + 13y^2$  & $4x^2 + 3xy + 7y^2$ & \begin{tabular}{@{}c@{}}$q-q^2+q^4-q^8+q^9-q^{13}-q^{14}-q^{18}+O(q^{19})$ \\ $q - q^7 - q^8 + q^9 - q^{14} - q^{17} + q^{25} + q^{26} + O(q^{28})$ \end{tabular}\\
 \hline
 $-127$ &  $x^2 + xy + 32y^2$ & $2x^2 + xy + 16y^2$  & $4x^2 + xy + 8y^2$ &\begin{tabular}{@{}c@{}}$q-q^2+q^4-q^8+q^9-q^{17}-q^{18}-q^{19}+O(q^{22})$ \\ $q -  q^{8} +  q^{9} -  q^{11} -  q^{13} -  q^{22} +  q^{25} -  q^{26} +  O(q^{32})$ \end{tabular}\\
 \hline
   $-131$ &  $x^2 + xy + 33y^2$ & $3x^2 + xy + 11y^2$  & $5x^2 + 3xy + 7y^2$ & \begin{tabular}{@{}c@{}}$q-q^3+q^4+q^9-q^{11}-q^{12}-q^{13}-q^{15}+O(q^{16})$ \\ $q +  q^{4} -  q^{5} -  q^{7} -  q^{15} +  q^{16} -  q^{20} -  q^{21} +  O(q^{25})$ \end{tabular} \\ 
 \hline
 $-179$ &  $x^2 + xy + 45y^2$ & $3x^2 + xy + 15y^2$  & $5x^2 + xy + 7y^2$ &\begin{tabular}{@{}c@{}}$q-q^3+q^4+q^9-q^{12}-q^{15}+q^{16}-q^{17}+O(q^{19})$ \\ $q +  q^{4} -  q^{5} -  q^{13} -  q^{15} +  q^{16} -  q^{20} +  q^{25} +  O(q^{27})$ \end{tabular}\\  
 \hline
 $-227$ &  $x^2 + xy + 57y^2$ & $3x^2 + xy + 19y^2$  & $7x^2 + 5xy + 9y^2$ &\begin{tabular}{@{}c@{}}$q-q^3+q^4+q^9-q^{12}+q^{16}-q^{19}-q^{21}+O(q^{23})$ \\ $q +  q^{4} -  q^{7} -  q^{11} +  q^{16} -  q^{21} +  q^{25} -  q^{27} +  O(q^{28})$ \end{tabular}\\ 
 \hline
 $-347$ &  $x^2 + xy + 87y^2$ & $3x^2 + xy + 29y^2$  & $9x^2 + 7xy + 11y^2$ &\begin{tabular}{@{}c@{}}$q-q^3+q^4+q^9-q^{12}+q^{16}+q^{25}-q^{27}+O(q^{29})$ \\ $q +  q^{4} -  q^{11} -  q^{13} +  q^{16} +  q^{25} -  q^{27} -  q^{33} +  O(q^{39})$ \end{tabular}\\
 \hline
 $-443$ &  $x^2 + xy + 111y^2$ & $3x^2 + xy + 37y^2$  & $9x^2 + 5xy + 13y^2$ &\begin{tabular}{@{}c@{}}$q-q^3+q^4+q^9-q^{12}+q^{16}+q^{25}-q^{27}+O(q^{36})$ \\ $q +  q^{4} -  q^{13} +  q^{16} -  q^{17} +  q^{25} -  q^{27} -  q^{39} +  O(q^{49})$ \end{tabular}\\ 
 \hline
 $-523$ & $x^2 + xy + 131y^2$ & $7x^2 + 3xy + 19y^2$  & $11x^2 + 7xy + 13y^2$ & \begin{tabular}{@{}c@{}}$q+q^4-q^7+q^9+q^{16}-q^{19}-q^{23}+q^{25}+O(q^{28})$ \\ $q +  q^{4} +  q^{9} -  q^{11} -  q^{13} +  q^{16} -  q^{17} +  q^{25} +  O(q^{31})$ \end{tabular} \\ 
 \hline
$-571$ & $x^2 + xy + 143y^2$  & $5x^2 + 3xy + 29y^2$ & $11x^2 + xy + 13y^2$ & \begin{tabular}{@{}c@{}}$q + q^4 - q^5 + q^9 + q^{16} -q^{20} +q^{25} -q^{29}  + O(q^{31})$ \\ $q +  q^{4} +  q^{9} -  q^{11} -  q^{13} +  q^{16} -  q^{23} +  q^{36}  +  O(q^{44})$ \end{tabular}\\ 
 \hline
 $-619$ & $x^2 + xy + 155y^2$ & $5x^2 + xy + 31y^2$  & $7x^2 + 5xy + 23y^2$ & 
  \begin{tabular}{@{}c@{}}$q + q^4 - q^5 +q^9 + q^{16} -q^{20} +q^{25} -q^{31}  + O(q^{35})$ \\ $q +  q^{4} -  q^{7} +  q^{9} +  q^{16} -  q^{23} -  q^{28} -  q^{35} +  O(q^{36})$ \end{tabular}\\ 
 \hline
  $-683$ & $x^2 + xy + 171y^2$ & $3x^2 + xy + 57y^2$  & $9x^2 + xy + 19y^2$ & 
   \begin{tabular}{@{}c@{}}$q - q^3 +q^4 +q^9 -q^{12} + q^{16} +q^{25} -q^{27} + O(q^{36})$ \\ $q +  q^{4} +  q^{16} -  q^{19} +  q^{25} -  q^{27} -  q^{29} +  q^{49}  +  O(q^{50})$ \end{tabular}\\ 
 \hline
 $-691$ & $x^2 + xy + 173y^2$ & $5x^2 + 3xy + 35y^2$  & $7x^2 + 3xy + 25y^2$  &
   \begin{tabular}{@{}c@{}}$q +q^4 -q^5 +q^9 +q^{16} -q^{20} +q^{25} -q^{35} + O(q^{36})$ \\ $q +  q^{4} -  q^{7} +  q^{9} +  q^{16} -  q^{28} -  q^{29} -  q^{35}  +  O(q^{36})$ \end{tabular}\\ 
 \hline
 $-739$ & $x^2 + xy + 185y^2$ & $5x^2 + xy + 37y^2$  & $11x^2 + 3xy + 17y^2$ &
   \begin{tabular}{@{}c@{}}$q +q^4 -q^5 +q^9 +q^{16} -q^{20} +q^{25} +q^{36} + O(q^{37})$ \\ $q +  q^{4} +  q^{9} -  q^{11} +  q^{16} -  q^{17} -  q^{31} +  q^{36} +  O(q^{44})$ \end{tabular}\\ 
 \hline
 $-787$ & $x^2 + xy + 197y^2$ & $7x^2 + 5xy + 29y^2$  & $11x^2 + 7xy + 19y^2$ & 
 \begin{tabular}{@{}c@{}}$q + q^4 - q^7 +q^9 +q^{16} +q^{25} -q^{28} -q^{29} + O(q^{31})$ \\ $q +  q^{4} +  q^{9} -  q^{11} +  q^{16} -  q^{19} -  q^{23} +  q^{25} +  O(q^{36})$ \end{tabular}\\ 
 \hline
 $-947$& $x^2 + xy + 237y^2$ & $3x^2 + xy + 79y^2$ & $9x^2 + 5xy + 27y^2$ &
  \begin{tabular}{@{}c@{}}$q-q^3+q^4 +q^9 -q^{12} +q^{16} +q^{25} -q^{27} + O(q^{36})$ \\ $ q + q^4 + q^{16} + q^{25} - q^{27} - q^{31} - q^{41} + q^{49}  +  O(q^{50})$ \end{tabular}\\ 
 \hline
 $-1051$ & $x^2 + xy + 263y^2$ & $5x^2 + 3xy + 53y^2$ & $11x^2 + 7xy + 25y^2$ & 
  \begin{tabular}{@{}c@{}}$q + q^4 - q^5 + q^9 +q^{16} -q^{20} +q^{25} +q^{36} + O(q^{45})$ \\ $ q + q^4 + q^9 - q^{11} + q^{16} - q^{29} + q^{36} - q^{43}  +  O(q^{49})$ \end{tabular}\\ 
 \hline
$-1123$ & $x^2 + xy + 281y^2$ & $7x^2 + 5xy + 41y^2$  & $17x^2 + 13xy + 19y^2$ &
  \begin{tabular}{@{}c@{}}$q+ q^4 -q^7 +q^9 +q^{16}+q^{25} - q^{28} + q^{36} + O(q^{41})$ \\ $ q + q^4 + q^9 + q^{16} - q^{17} - q^{19} - q^{23} + q^{25}   +  O(q^{36})$ \end{tabular}\\ 
 \hline
 $-1723$ & $x^2 + xy + 431y^2$  & $11x^2 + 9xy + 41y^2$ & $19x^2 + 5xy + 23y^2$ &  \begin{tabular}{@{}c@{}}$q + q^4 + q^9 -q^{11} +q^{16} +q^{25} +q^{36}  -q^{41} + O(q^{43})$ \\ $ q+ q^4 + q^9 +q^{16} -q^{19} -q^{23} +q^{25} +q^{36}  +  O(q^{37})$ \end{tabular}\\ 
 \hline
  $-1747$ & $x^2 + xy + 437y^2$ & $17x^2 + 15xy + 29y^2$ & $19x^2 + xy + 23y^2$ &\begin{tabular}{@{}c@{}}$q + q^4 + q^9 + q^{16} -q^{17} +q^{25} -q^{29} -q^{31} + O(q^{36})$ \\ $ q + q^4 + q^9 + q^{16} -q^{19} -q^{23} +q^{25} +q^{36} +  O(q^{41})$ \end{tabular}\\ 
 \hline
  $-1867$   & $x^2 + xy + 467y^2$ & $7x^2 + 3xy + 67y^2$  & $11x^2 + 5xy + 43y^2$ &
  \begin{tabular}{@{}c@{}}$q + q^4 - q^7 + q^9 +q^{16} + q^{25} -q^{28} + q^{36}  + O(q^{49})$ \\ $ q + q^4 + q^9 -q^{11} +q^{16} +q^{25} + q^{36} -q^{43} + O(q^{44})$ \end{tabular}\\ 
 \hline
   $-2203$ & $x^2 + xy + 551y^2$   & $7x^2 + 3xy + 79y^2$ & $19x^2 + xy + 29y^2$ &
   \begin{tabular}{@{}c@{}}$q + q^4 -q^7 + q^9 +q^{16} + q^{25} -q^{28} +q^{36} + O(q^{49})$ \\ $q + q^4 +q^9 +q^{16} -q^{19} +q^{25} -q^{29} + q^{36}  +  O(q^{47})$ \end{tabular}\\ 
 \hline
  $-2347$ & $x^2 + xy + 587y^2$   & $17x^2 + 13xy + 37y^2$ & $19x^2 + 3xy + 31y^2$ &
  \begin{tabular}{@{}c@{}}$q+ q^4 + q^9 + q^{16} -q^{17} +q^{25} +q^{36} -q^{37}  + O(q^{41})$ \\ $ q+ q^4 + q^9 + q^{16} - q^{19} +q^{25} -q^{31} +q^{36} +  O(q^{47})$ \end{tabular}\\ 
 \hline
  $-2683$ & $x^2 + xy + 671y^2$ & $11x^2 + xy + 61y^2$ & $23x^2 + 13xy + 31y^2$ &  \begin{tabular}{@{}c@{}}
   $q + q^4 + q^{9} -q^{11} +q^{16} +q^{25} +q^{36} -q^{44} + O(q^{49})$ \\ $ q +q^4 +q^9 +q^{16} -q^{23} +q^{25} -q^{31} +q^{36}+  O(q^{41})$ \label{C} \end{tabular}\\ 
 \hline

 \end{tabular} \caption{Class number 5} \label{C}
\end{table}

\end{itemize}

\newpage
\subsection*{Acknowledgement.} The authors would like to thank Pieter Moree for comments and useful correspondences related to this work.

\begin{rezabib} 

\bib{ARW}{article}{
   author={Arno, Steven},
   author={Robinson, M. L.},
   author={Wheeler, Ferrell S.},
   title={Imaginary quadratic fields with small odd class number},
   journal={Acta Arith.},
   volume={83},
   date={1998},
   number={4},
   pages={295--330},
   issn={0065-1036},
   review={\MR{1610549}},
   doi={10.4064/aa-83-4-295-330},
}

\bib{B}{article}{
   author={Bhattacharya, Soumya},
   title={Finiteness of simple holomorphic eta quotients of a given weight},
   journal={Adv. Math.},
   volume={308},
   date={2017},
   pages={879--895},
   issn={0001-8708},
   review={\MR{3600077}},
   doi={10.1016/j.aim.2016.12.010},
}
\bib{VDB}{article}{
   author={van der Blij, F.},
   title={Binary quadratic forms of discriminant $-23$.},
   journal={Nederl. Akad. Wetensch. Proc. Ser. A. {\bf 55} = Indagationes
   Math.},
   date={1952},
   pages={498--503},
   review={\MR{0052462}},
}

\bib{HC}{book}{
   author={Cohen, Henri},
   author={Str\"{o}mberg, Fredrik},
   title={Modular forms},
   series={Graduate Studies in Mathematics},
   volume={179},
   note={A classical approach},
   publisher={American Mathematical Society, Providence, RI},
   date={2017},
   pages={xii+700},
   isbn={978-0-8218-4947-7},
   review={\MR{3675870}},
   doi={10.1090/gsm/179},
}

\bib{DC}{book}{
   author={Cox, David A.},
   title={Primes of the form $x^2+ny^2$---Fermat, class field theory, and
   complex multiplication},
   edition={3},
   note={With contributions by Roger Lipsett},
   publisher={AMS Chelsea Publishing, Providence, RI},
   date={[2022] \copyright 2022},
   pages={xv+533},
   isbn={[9781470470289]},
   isbn={[9781470471835]},
   review={\MR{4502401}},
}

\bib{FD}{book}{
   author={Diamond, Fred},
   author={Shurman, Jerry},
   title={A first course in modular forms},
   series={Graduate Texts in Mathematics},
   volume={228},
   publisher={Springer-Verlag, New York},
   date={2005},
   pages={xvi+436},
   isbn={0-387-23229-X},
   review={\MR{2112196}},
}

\bib{ES}{article}{
   author={Eholzer, Wolfgang},
   author={Skoruppa, Nils-Peter},
   title={Product expansions of conformal characters},
   journal={Phys. Lett. B},
   volume={388},
   date={1996},
   number={1},
   pages={82--89},
   issn={0370-2693},
   review={\MR{1418608}},
   doi={10.1016/0370-2693(96)01154-9},
}

\bib{H}{article}{
   author={Huber, Tim},
   author={Liu, Chang},
   author={McLaughlin, James},
   author={Ye, Dongxi},
   author={Yuan, Miaodan},
   author={Zhang, Sumeng},
   title={On the vanishing of the coefficients of CM eta quotients},
   journal={Proc. Edinb. Math. Soc. (2)},
   volume={66},
   date={2023},
   number={4},
   pages={1202--1216},
   issn={0013-0915},
   review={\MR{4679221}},
   doi={10.1017/s0013091523000627},
}

\bib{Kob}{book}{
   author={Koblitz, Neal},
   title={Introduction to elliptic curves and modular forms},
   series={Graduate Texts in Mathematics},
   volume={97},
   edition={2},
   publisher={Springer-Verlag, New York},
   date={1993},
   pages={x+248},
   isbn={0-387-97966-2},
   review={\MR{1216136}},
   doi={10.1007/978-1-4612-0909-6},
}


\bib{O}{book}{
   author={Ono, Ken},
   title={The web of modularity: arithmetic of the coefficients of modular
   forms and $q$-series},
   series={CBMS Regional Conference Series in Mathematics},
   volume={102},
   publisher={Conference Board of the Mathematical Sciences, Washington, DC;
   by the American Mathematical Society, Providence, RI},
   date={2004},
   pages={viii+216},
   isbn={0-8218-3368-5},
   review={\MR{2020489}},
}

\bib{sage}{manual}{
      author={Developers, The~Sage},
       title={{S}agemath, the {S}age {M}athematics {S}oftware {S}ystem
  ({V}ersion 5.46.0)},
        date={2023},
}

\bib{JPS}{book}{
   author={Serre, Jean-Pierre},
   title={Lectures on $N_X (p)$},
   series={Chapman \& Hall/CRC Research Notes in Mathematics},
   volume={11},
   publisher={CRC Press, Boca Raton, FL},
   date={2012},
   pages={x+163},
   isbn={978-1-4665-0192-8},
   review={\MR{2920749}},
}

\bib{Sch}{article}{
   author={Schoeneberg, B.},
   title={Bemerkungen \"{u}ber einige Klassen von Modulformen},
   note={Nederl. Akad. Wetensch. Proc. Ser. A {\bf 70}},
   language={German},
   journal={Indag. Math.},
   volume={29},
   date={1967},
   pages={177--182},
   review={\MR{0210674}},
}

\bib{stark}{article}{
   author={Stark, H. M.},
   title={Values of $L$-functions at $s=1$. I. $L$-functions for quadratic
   forms},
   journal={Advances in Math.},
   volume={7},
   date={1971},
   pages={301--343 (1971)},
   issn={0001-8708},
   review={\MR{0289429}},
   doi={10.1016/S0001-8708(71)80009-9},
}

%
%

\bib{W}{article}{
   author={Walling, Lynne H.},
   title={A remark on differences of theta series},
   journal={J. Number Theory},
   volume={48},
   date={1994},
   number={2},
   pages={243--251},
   issn={0022-314X},
   review={\MR{1285542}},
   doi={10.1006/jnth.1994.1065},
}

\bib{XWDP}{book}{
   author={Wang, Xueli},
   author={Pei, Dingyi},
   title={Modular forms with integral and half-integral weights},
   publisher={Science Press Beijing, Beijing; Springer, Heidelberg},
   date={2012},
   pages={x+432},
   isbn={978-7-03-033079-6},
   isbn={978-3-642-29301-6},
   isbn={978-3-642-29302-3},
   review={\MR{3015123}},
   doi={10.1007/978-3-642-29302-3},
}

\bib{Wa}{article}{
   author={Watkins, Mark},
   title={Class numbers of imaginary quadratic fields},
   journal={Math. Comp.},
   volume={73},
   date={2004},
   number={246},
   pages={907--938},
   issn={0025-5718},
   review={\MR{2031415}},
   doi={10.1090/S0025-5718-03-01517-5},
}

\bib{Z}{book}{
   author={Zagier, D. B.},
   title={Zetafunktionen und quadratische K\"orper},
   language={German},
   series={Hochschultext. [University Textbooks]},
   note={Eine Einf\"uhrung in die h\"ohere Zahlentheorie. [An introduction
   to higher number theory]},
   publisher={Springer-Verlag, Berlin-New York},
   date={1981},
   pages={viii+144},
   isbn={3-540-10603-0},
   review={\MR{0631688}},
}

\end{rezabib}

\end{document}